\documentclass[11pt]{amsart}
\usepackage{mathrsfs,latexsym,amsfonts,amssymb}
\usepackage{hyperref}
\usepackage{todonotes}
\newtheorem{theorem}{Theorem}[section]
\newtheorem{lemma}[theorem]{Lemma}
\newtheorem{corollary}[theorem]{Corollary}
\newtheorem{question}[theorem]{Question}
\newtheorem{example}[theorem]{Example}
\theoremstyle{definition}
\newtheorem{definition}[theorem]{Definition}
\newtheorem{proposition}[theorem]{Proposition}
\theoremstyle{remark}

\newcommand{\Aut}{\operatorname{Aut}}
\newcommand{\gyr}{\operatorname{gyr}}

\begin{document}
\title[Suitable sets for strongly topological gyrogroups]
{Suitable sets for strongly topological gyrogroups}

\author{Fucai Lin*}
\address{Fucai Lin: School of mathematics and statistics,
Minnan Normal University, Zhangzhou 363000, P. R. China}
\email{linfucai2008@aliyun.com; linfucai@mnnu.edu.cn}

\author{Tingting Shi}
\address{Tingting Shi: School of mathematics and statistics,
Minnan Normal University, Zhangzhou 363000, P. R. China}
\email{277653220@qq.com}

\author{Meng Bao}
\address{Meng Bao: School of mathematics and statistics,
Minnan Normal University, Zhangzhou 363000, P. R. China}
\email{mengbao95213@163.com}

\thanks{The first author is supported by the NSFC (No.11571158), the Natural Science Foundation of
Fujian Province (No. 2017J01405) of China, the Program for New Century Excellent Talents in Fujian
Province University, the Institute of Meteorological Big Data-Digital Fujian and Fujian Key Laboratory
of Data Science and Statistics. The second author is supported by the Young and middle-aged project
in Fujian Province (No. JAT190397).}

\thanks{* Corresponding author}

\keywords{suitable sets; gyrogroups; (strongly) topological gyrogroups; left precompact; pseudocompact; separable space; metrizable space.}
\subjclass[2020]{Primary 54H11; 22A05, Secondary 20B30; 20N99; 54A25; 54C35; 54D65; 54E35.}

\begin{abstract}
A discrete subset $S$ of a topological gyrogroup $G$ with the identity $0$ is said to be a
{\it suitable set} for $G$ if it generates a dense subgyrogroup of $G$ and $S\cup \{0\}$
is closed in $G$. In this paper, it was proved that each countable Hausdorff topological
gyrogroup has a suitable set; moreover, it is shown that each separable metrizable strongly topological
gyrogroup has a suitable set.
\end{abstract}

\maketitle
\section{Introduction}
In 1990, K.H. Hofmann and S.A. Morris in \cite{HM} introduced
the concept of a suitable set for a
topological group as an example of a 'thin' closed generating set. It was shown that
each locally compact group has a suitable set.
Fundamental results on suitable sets for topological groups
were obtained by Comfort et al. in~\cite{CMRS1998}
and Dikranjan et al. in~\cite{DTT1999} and ~\cite{DTT2000}.
I. Guran in~\cite{G2003}, F. Lin, A. Ravsky, and T. Shi in~\cite{LR}
considered suitable sets for paratopological groups.
In 2003, T. Banakh and I. Protasov generalized Guran's results to left topological groups in~\cite{BP2003}.

A generalization of a group, gyrogroup (see Definition~\ref{def:gyrogroup} below)
was introduced by A.A. Ungar~\cite{UA2002} in
2002, while studying a $c$-ball $\mathbb{R}_{c}^{3}=
\{\mathbf{v}\in\mathbb{R}^{3}: \|\mathbf{v}\|<c\}$
of relativistically admissible velocities
endowed with Einstein velocity addition $\oplus_{E}$.
Recall that for vectors $\mathbf{u}, \mathbf{v}\in \mathbb{R}_{c}^{3}$
$$\mathbf{u}\oplus_{E}\mathbf{v}=\frac{1}{1+\frac{\langle \mathbf{u}, \mathbf{v}\rangle}{c^{2}}}
\left(\mathbf{u}+\frac{1}{\gamma_{\mathbf{u}}}\mathbf{v}+\frac{1}{c^{2}}\frac{\gamma_{\mathbf{u}}}{1+\gamma_{\mathbf{u}}}\langle\mathbf{u}, \mathbf{v}\rangle\mathbf{u}\right),$$
where $$\gamma_{\mathbf{u}}=\frac{1}{\sqrt{1-\frac{\parallel\mathbf{u}\parallel^{2}}{c^{2}}}}$$
is the Lorentz factor.
It turned out that $(\mathbb{R}_{c}^{3}, \oplus_{E})$ is a gyrogroup,
which fails to be a group, because the operation $\oplus_{E}$ is not associative.
Recently, the topic of gyrogroups was investigated by many scholars, see~\cite{FM,FM1,LF,LF1,LF2,LF3,SL,RS,ST,ST1,ST2,UA2002}.

In 2017, W. Atiponrat~\cite{AW} introduced the concept of topological gyrogroups, which is
a generalization of a topological group. Namely, a topological gyrogroup $G$ is a gyrogroup
$(G,\oplus)$ endowed with a topology such that
the multiplication map $\oplus$ from $G\times G$ to $G$ is jointly continuous
and the inverse map $\ominus:G\rightarrow G$ is continuous.
In turned out that topological gyrogroups possess nice properties.
In particular, Z. Cai, S. Lin, and W. He in \cite{CZ} proved that every topological gyrogroup is a
rectifiable space, so every first-countable topological gyrogroup is metrizable. Then R. Shen
in~\cite{RS} proved that every weakly first-countable paratopological left-loop is first-countable.
M. Bao and F. Lin introduced the concept of strongly topological gyrogroups, and proved that
every feathered strongly topological gyrogroup is paracompact, every $T_{0}$ strongly
topological gyrogroup is completely regular and every $T_{0}$ strongly topological gyrogroup with
a countable pseudocharacter is submetrizable, see \cite{BL,BL1,BL2}.

In this paper, we mainly consider suitable sets for (strongly) topological gyrogroups.
A subset $S$ of a topological gyrogroup $G$ is said to be a {\it suitable set} for
$G$ if (1) $\overline{\langle S\rangle}=G$, (2) $S$ has the discrete topology, and (3) $S\cup \{0\}$
is closed in $G$. We show that each countable Hausdorff topological
gyrogroup has a suitable set, and each separable metrizable strongly topological gyrogroup has a suitable set, which generalizes some results for topological groups in \cite{CMRS1998,DTT1999}.

All spaces throughout this paper are supposed to be Hausdorff, unless the opposite is not stated.
Let $\mathbb{N}$ be the set of all positive integers and $\omega$ the first infinite ordinal.
Let $X$ be a topological space, and let $A$ be a subset of $X$.
The {\it closure} of $A$ in $X$ is denoted by $\overline{A}$.
For undefined notation and terminology, the reader may refer to~\cite{AA, E}.

\section{Motivation and Preliminaries}
In this section, we provide a motivation to study suitable sets in topological gyrogroups.
Also we recall and introduce notions and notation used in the paper.

\begin{definition}\cite{AW} A {\it groupoid} is a pair $(G, \oplus)$,
where $G$ is a nonempty set and $\oplus$ is a binary operation on $G$.
A function $f$ from a groupoid $(G_{1}, \oplus_{1})$ to a groupoid $(G_{2}, \oplus_{2})$ is called a
{\it groupoid homomorphism}, if $f(x\oplus_{1}y)=f(x)\oplus_{2} f(y)$ for any elements $x, y\in G_{1}$.
Furthermore, a bijective groupoid homomorphism from a groupoid $(G, \oplus)$ to itself will be called a
{\it groupoid automorphism}. We denote for a set of all automorphisms
of a groupoid $(G, \oplus)$ by $\Aut(G, \oplus)$.
\end{definition}

\begin{definition}\cite{UA}\label{def:gyrogroup}
A groupoid $(G,\oplus)$ is called a {\it gyrogroup}, if its binary operation satisfies the following
conditions.

\smallskip
(G1) There exists a unique identity element $0\in G$ such that $0\oplus a=a=a\oplus0$ for all $a\in G$.

\smallskip
(G2) For each $x\in G$, there exists a unique inverse element $\ominus x\in G$ such that $\ominus x \oplus x=0=x\oplus (\ominus x)$.

\smallskip
(G3) There exists a map $\gyr:G\times S\to \Aut(G, \oplus)$,
such that $x\oplus (y\oplus z)=(x\oplus y)\oplus\gyr[x, y](z)$ for all $z\in G$.

\smallskip
(G4) For any $x, y\in G$, $\mbox{gyr}[x\oplus y, y]=\mbox{gyr}[x, y]$.
\end{definition}

\begin{definition}\cite{ST}
A nonempty subset $H$ of a gyrogroup $(G,\oplus)$ is called a {\it subgyrogroup} of $G$ (denoted
by $H\leq G$), provided the following conditions hold.

$(i)$ The restriction $\oplus| _{H\times H}$ is a binary operation on $H$, i.e. $(H, \oplus| _{H\times H})$ is a groupoid.

\smallskip
$(ii)$ For any $x, y\in H$, the restriction of $\gyr[x, y]$ to $H$, $\gyr[x, y]|_{H}$ : $H\rightarrow \gyr[x, y](H)$, is a bijective homomorphism.

\smallskip
$(iii)$ $ (H, \oplus|_{H\times H})$ is a gyrogroup.

\smallskip
A subgyrogroup $H$ of $G$ is said to be an {\it $L$-subgyrogroup} \cite{ST}, denoted
by $H\leq_{L} G$, if $\gyr[a, h](H)=H$ for all $a\in G$ and $h\in H$.
\end{definition}

\begin{definition}\cite{AW}
A triple $(G, \tau, \oplus)$ is called a {\it topological gyrogroup},
provided the following conditions hold.

\smallskip
(1) $(G, \tau)$ is a topological space.

\smallskip
(2) $(G, \oplus)$ is a gyrogroup.

\smallskip
(3) The binary operation $\oplus: G\times G\rightarrow G$ is jointly continuous, where $G\times G$
is endowed with the product topology and the inversion
$\ominus: G\rightarrow G$, $x\mapsto \ominus x$, is continuous.
\end{definition}

It is easy to see that each topological group is a topological gyrogroup $(G,\tau,\oplus)$
provided we put $\gyr[x,y](z)=z$ all $x,y,z \in G$.
A well-known example of a topological gyrogroup, which is not a topological group,
is the following {\it M\"{o}bius topological gyrogroup}.

\begin{example}\cite{AW}\label{lz1}
Let $D$ be a open unit disk $\{z\in\mathbb C:|z|<1\}$ in the complex plane.
Define a M\"{o}bius addition $\oplus _{M}: D\times D\rightarrow D$ putting
$$a\oplus _{M}b=\frac{a+b}{1+\bar{a}b}\ \mbox{for all}\ a, b\in D.$$
Then $(D, \oplus _{M})$ is a gyrogroup with
$$\gyr[a, b](c)=\frac{1+a\bar{b}}{1+\bar{a}b}c\ \mbox{for any}\ a, b, c\in D.$$
But $(D, \oplus _{M})$ is not a group, because the operation $\oplus _{M}$ is not associative.
Indeed, it is easy to check that $(1/2 \oplus _{M} i/2) \oplus _{M} (-1/2)\ne 1/2 \oplus _{M} (i/2 \oplus _{M} (-1/2))$.
If $\tau$ the usual topology on $D$ then $(D,\tau,\oplus _{M})$ is a topological gyrogroup.
\end{example}

\begin{definition}\cite{BL}
A topological gyrogroup $G$ is a {\it strongly topological gyrogroup} if
there exists a neighborhood base $\mathscr U$ of $0$ such that
$\gyr[x, y](U)=U$ for each $x, y\in G$ and $U\in \mathscr U$.
In this case we shall say that $G$ is a strongly topological gyrogroup with a neighborhood base
$\mathscr U$ of $0$. Clearly, we may assume that $U$ is symmetric for each $U\in\mathscr U$.
\end{definition}

We claim that $(D,\tau,\oplus _{M})$ in Example \ref{lz1} is a strongly topological gyrogroup
\cite{BL}. Indeed, for any $n\in\omega$, let $U_{n}=\{x\in D: |x|\leq \frac{1}{n}\}$.
Then, $\mathscr U=\{U_{n}: n\in \omega\}$ is a neighborhood base of $0$. Moreover,
since $\overline{1+a\bar{b}}=1+\bar{a}b$ for each $a,b\in D$, we have
$|\frac{1+a\bar{b}}{1+\bar{a}b}|=1$. Therefore, we see that
$\gyr[x, y](U)\subset U$, for any $x, y\in D$ and each $U\in \mathscr U$.
By \cite[Proposition 2.6]{ST} it follows that $\gyr[x, y](U)=U$.

Moreover, M\"{o}bius gyrogroups, Einstein gyrogroups, and Proper velocity gyrogroups,
that were studied in \cite{FM, FM1,UA}, are all strongly topological gyrogroups, see \cite{BL}.

\begin{definition}\cite{HM}
Let $G$ be a topological gyrogroup and $S$ a subset of $G$. Then $S$ is said to be a {\it suitable set}
 for $G$ if $S$ is discrete in itself, generates a dense subgyrogroup of $G$, and $S\cup \{0\}$ is closed in $G$.
\end{definition}

\smallskip
By the same notations of \cite{DTT1999}, let $\mathcal{S}$ (resp., $\mathcal{S}_{c}$) be the class of topological gyrogroups
having a suitable (resp., closed suitable) set. It turns out that very often the subset $S$ of the group $G$ has the stronger
property to generate $G$, instead of generating just a dense subgroup of $G$. We denote by
$\mathcal{S}_{g}$ and $\mathcal{S}_{cg}$ the corresponding subclasses of $\mathcal{S}$ and $\mathcal{S}_{c}$, respectively.

The following proposition generalizes \cite[Proposition 1.4]{CMRS1998}.

\begin{proposition}
If a topological gyrogroup $(G,\oplus)$ has a suitable set, then $G$ is Hausdorff or $|G|\leq 2$.
\end{proposition}

\begin{proof}
Assume that $G$ is not Hausdorff and $|G|\geq 3$. Let $S$ be the suitable set for $G$. Since $G$ is not Hausdorff and $T_{0}$ and $T_{3}$ are equivalent in
topological gyrogroups by \cite{AW}, for every $g\in G$ a set $\overline{\{g\}}$
contains a point $h\not =g$.
By the assumption of $|G|\geq 3$, it follows that $S\cup\{0\}$ has at least two points. Take an arbitrary point $s\in S\setminus\{0\}$. Since $S$ is discrete in itself, we have $S\cap \overline{\{s\}}=\{s\}$. Further, $S\cup \{0\}$ is closed in $G$, thus $\overline{\{s\}}=\{s, 0\}$. Therefore, $\overline{\{0\}}=\{s, 0\}$. It follows that $S$ has at most two points, $s$ and $0$. Moreover, since $\overline{\{0\}}$ is a gyrogroup, it is clear that $s\oplus s=0$. Then, $s=\ominus s=0$, that is, $G=S$, this is a contradiction.
\end{proof}

Recall that given a space $X$, a {\it pseudocharacter $\psi(X)$} of $x$ is the smallest
infinite cardinal $\kappa$ such that any point of $X$ is an intersection of at most $\kappa$
open subsets of $X$ and {\it extent $e(X)$}
is the supremum of cardinalities of closed discrete subspaces of $X$.
Similarly to the proof of~\cite[Lemma 2.3]{DTT1999}, we can show the following

\begin{proposition}
A Hausdorff topological gyrogroup $G$ which has a suitable set satisfies $d(G)\leq e(G)\cdot \psi (G)$.
\end{proposition}

\begin{proof}
We assume that $A$ is a suitable set for $G$. If $U$ is an open neighborhood of $0$ in $G$, then $A\setminus U$ is discrete and closed in $G$, which implies  $|A\setminus U|\leq e(G)$. Pick a family $\gamma$ of open sets in $G$ such that $\bigcap \gamma =\{0\}$ and $|\gamma |=\psi (G)$. Since $A\setminus \{0\}\subset \bigcup \{A\setminus U: U\in \gamma\}$, it follows that $|A|\leq e(G)\cdot \psi (G)$. The subgyrogroup $H=\langle A\rangle$ of $G$ satisfies $|H|\leq |A|\cdot \aleph _{0}$. Since $A$ is a suitable set and $H$ is dense in $G$, we can conclude that $$d(G)\leq |H|\leq |A|\cdot \aleph _{0}\leq e(G)\cdot \psi (G).$$
\end{proof}

Therefore, it is natural to have the following result.

\begin{corollary}\label{c2}
A non-separable Lindel\"{o}f Hausdorff topological gyrogroup of countable pseudocharacter does not have a suitable set.
\end{corollary}

\begin{example}
There exists a non-separable Lindel\"{o}f Hausdorff topological gyrogroup $G$ of countable pseudocharacter such that $G$ does not have a suitable set and $G$ is not a topological group.
\end{example}

\begin{proof}
Let $D$ be the topological gyrogroup in Example~\ref{lz1}, and let $H$ be the Lindel\"{o}f non-separable topological group with countable pseudocharacter in (a) of \cite[Theorem 2.4]{DTT1999}. Then $D$ has a suitable set by in the following Corollary~\ref{c1} and $H$ does not have any suitable set. Moreover, the product $G=D\times H$ is a Lindel\"{o}f non-separable topological gyrogroup with countable pseudocharacter, hence it does not have any suitable set by Corollary~\ref{c2}. Clearly, $G$ is not a topological group.
\end{proof}

In this paper we mainly consider the following question.

\begin{question}
If $G$ belongs to some class $\mathcal{C}$ of Hausdorff topological gyrogroups, does $G$ have a suitable set?
\end{question}

\section{Countable topological gyrogroup with a suitable set}
In this section, we study the suitable sets in the class $\mathcal{C}$ of Hausdorff countable topological gyrogroups. We prove that every Hausdorff countable topological gyrogroup $G$ has a closed discrete subset $S$ such that $\langle S\rangle =G$. First, we need some lemmas.

Let $G$ be a gyrogroup. Fix an $n\in\mathbb{N}$. For any $x_{1},\cdots, x_{n}\in G$ and $\varepsilon_{1}, \cdots, \varepsilon_{n}\in\{-1, 1\}$, denote by $R[\varepsilon_{1}x_{1}, \cdots, \varepsilon_{n}x_{n}]$ the set of all elements which is added some brackets in the summand $\varepsilon_{1}x_{1}\oplus \cdots \oplus\varepsilon_{n}x_{n}$ such that the summand belongs to $G$, where
$$\varepsilon _{i}x_{i}=\left\{
\begin{matrix}
x_{i}, &\varepsilon _{i}=1;\\
\ominus x_{i}, &\varepsilon _{i}=-1.
\end{matrix}
\right.$$ Clearly, $R[\varepsilon_{1}x_{1}, \cdots, \varepsilon_{n}x_{n}]$ is a countable set, so enumerate $R[\varepsilon_{1}x_{1}, \cdots, \varepsilon_{n}x_{n}]$ as $$\{f_{m}(\varepsilon_{1}x_{1}, \cdots, \varepsilon_{n}x_{n}): m\in\mathbb{N}\}.$$If $A_{1}, \cdots, A_{n}\subset G$, then we denote $R[\varepsilon_{1}A_{1}, \cdots, \varepsilon_{n}A_{n}]$ and $f_{m}(\varepsilon_{1}A_{1}, \cdots, \varepsilon_{n}A_{n})$ as the sets $$\bigcup_{x_{1}\in A_{1}, \cdots, x_{n}\in A_{n}}R[\varepsilon_{1}x_{1}, \cdots, \varepsilon_{n}x_{n}]\ \mbox{and}\ \bigcup_{x_{1}\in A_{1}, \cdots, x_{n}\in A_{n}}f_{m}(\varepsilon_{1}x_{1}, \cdots, \varepsilon_{n}x_{n}),$$respectively.

In the class of topological gyrogroups, since the multiplication is jointly continuous and the inverse is continuous, it is easy to prove the following lemma.

\begin{lemma}\label{l11}
Let $a_{1}, a_{2}, \ldots, a_{n}$ be points of a topological gyrogroup $G$, and let $V$ be a neighborhood of the point $f_{m}(\varepsilon_{1}a_{1}, \cdots, \varepsilon_{n}a_{n})$. Then there exists neighborhoods $U_{1}, \ldots, U_{n}$ of $a_{1}, \ldots, a_{n}$ in $G$ respectively such that $f_{m}(\varepsilon_{1}U_{1}, \cdots, \varepsilon_{n}U_{n})\subset V$.
\end{lemma}

A topological space $X$ is {\it zero-dimensional} if it has a base consisting of clopen subsets.

\begin{lemma}\label{yl1}
Let $G$ be a nondiscrete Hausdorff topological gyrogroup and $U$ a nonempty open subset which generates $G$. Then every point $x\in U$ has an open neighborhood $V_{x}$ of $x$ such that $V_{x}\subset U$ and $\langle U\setminus \overline{V_{x}}\rangle =G$. In particular, if $G$ is zero-dimensional, then $V_{x}$ can be chosen to be clopen in $G$.
\end{lemma}

\begin{proof}
Let $U$ be a nonempty open subset which generates $G$. Take an arbitrary point $x\in U$. Since $G$ is not discrete, it is obvious that $U\setminus \{x\}$ is dense in $U$, then it follows that $\langle U\setminus \{x\}\rangle$ is dense in $\langle U\rangle =G$. Moreover, since $U\setminus \{x\}$ is open in $G$ and every open subgyrogroup is closed in $G$ by \cite[Proposition 7]{AW}, we can conclude that $\langle U\setminus \{x\} \rangle$ is open and closed in $G$. Therefore, $\langle U\setminus \{x\} \rangle =G$.

Since $x\in \langle U\setminus \{x\} \rangle$, there exist $y_{1}, y_{2},\ldots, y_{n}\in U\setminus \{x\}$, $\varepsilon _{1},\varepsilon _{2},\ldots ,\varepsilon _{n}\in \{1,-1\}$ and $m\in \mathbb{N}$ such that $x=f_{m}(\varepsilon_{1}y_{1}, \cdots, \varepsilon_{n}y_{n})$, where $$\varepsilon _{i}y_{i}=\left\{
\begin{matrix}
y_{i}, &\varepsilon _{i}=1;\\
\ominus y_{i}, &\varepsilon _{i}=-1.
\end{matrix}
\right.$$ Because each $y_{i}\neq x$, we can find an open neighborhood $O$ of $x$ such that $y_{i}\not \in \overline{O}\subset U$, for $i=1,\ldots ,n$. Then for each $i\in \{1,\ldots ,n\}$, there is an open neighborhood $O_{i}$ of $y_{i}$ such that $O_{i}\subset U$, $O\cap O_{i}=\emptyset$, and $f_{m}(\varepsilon_{1}O_{1}, \cdots, \varepsilon_{n}O_{n})\subset O$ by Lemma~\ref{l11}. Put $W=f_{m}(\varepsilon_{1}O_{1}, \cdots, \varepsilon_{n}O_{n})$. Then $W$ is an open neighborhood of $x$. By the regularity of $G$, there exists an open neighborhood $V_{x}$ of $x$ such that $\overline{V_{x}}\subset W\subset O$. Therefore, $O_{i}\subset U\setminus O\subset U\setminus \overline{V_{x}}$, for $i=1, 2, \ldots, n$. So, $\overline{V_{x}}\subset W\subset \langle U\setminus \overline{V_{x}}\rangle$. Thus $\langle U\setminus \overline{V_{x}}\rangle =\langle U\rangle =G$.

It is obvious that the last statement of this lemma.
\end{proof}

Now we can prove our main theorem in this section.

\begin{theorem}
Every countable Hausdorff topological gyrogroup $G$ belongs to $\mathcal{S}_{cg}$.
\end{theorem}

\begin{proof}
If $G$ is finitely generated or discrete, then the theorem is clear. Therefore, we may suppose that $G$ is neither finitely generated nor discrete. Enumerate $G$ as $\{g_{n}: n<\omega \}$. It suffice to find a subset $S$ in $G$ and an open neighborhood $U_{n}$ of $g_{n}$, for each $n<\omega $ satisfying $\langle S\rangle =G$ and $U_{n}\cap S$ is finite.

Next we will by induction to find a clopen set $V_{n}$ in $G$ and a finite set $S_{n}\subset G$ for each $n<\omega $ so that the following conditions hold:

\smallskip
(i) $g_{n}\in \bigcup_{i=0}^{n}V_{i}$ for each $n\in\omega$;

\smallskip
(ii) $G=\langle G\setminus (\bigcup_{i=0}^{n}V_{i})\rangle$ for each $n\in\omega$;

\smallskip
(iii) for each $n>0$, $V_{n}\subset G\setminus (\bigcup_{i=0}^{n-1}V_{i})$;

\smallskip
(iv) $V_{i}\cap S_{n}=\emptyset $, for $i<n$; and

\smallskip
(v) $g_{n}\in \langle \bigcup_{i=0}^{n}S_{i}\rangle$ for each $n\in\omega$.

\smallskip
Then set $U_{n}=\bigcup_{i=0}^{n}V_{i}$ for each $n\in\omega$, and put $S=\bigcup _{n<\omega}S_{n}$. Clearly, $\langle S\rangle =G$ and $U_{n}\cap S$ is finite for each $n\in\omega$.

Therefore, it suffices to construct $S_{n}$ and $V_{n}$ inductively as follows.

Set $S_{0}=\{g_{0}\}$. Since every Hausdorff topological gyrogroup is regular and every countable non-empty regular space is zero-dimensional \cite[Corollary 6.2.8]{E}, it follows that the countable topological gyrogroup $G$ is zero-dimensional. By Lemma \ref{yl1}, there exists a clopen neighborhood $V_{0}$ of $g_{0}$ such that $G=\langle G\setminus V_{0}\rangle$.

Assume that the finite sets $S_{0}, S_{1}, \ldots, S_{k}$ and clopen sets $V_{0}, V_{1}, \ldots, V_{k}$ have been defined satisfying the above properties (i)-(v). Clearly, if $g_{k+1}\in \langle \bigcup_{i=0}^{k}S_{i}\rangle$, then set $S_{k+1}=\emptyset$. If $g_{k+1}\not \in \langle \bigcup_{i=0}^{k}S_{i}\rangle$, then it follows from (ii) that there exist $$y_{1},y_{2},\ldots ,y_{m}\in G\setminus (\bigcup_{i=0}^{k}V_{i}), \varepsilon _{1},\varepsilon _{2},\ldots ,\varepsilon _{m}\in \{1,-1\}$$ and $n\in\mathbb{N}$ such that $$g_{k+1}=f_{n}(\varepsilon_{1}y_{1}, \cdots, \varepsilon_{m}y_{m}).$$ Set $S_{k+1}=\{y_{1},y_{2},\ldots ,y_{m}\}$. Thus both (iv) and (v) are satisfied.

Obviously, if $g_{k+1}\in \bigcup_{i=0}^{k}V_{i}$, then put $V_{k+1}=\emptyset$. If $g_{k+1}\not \in \bigcup_{i=0}^{k}V_{i}$, then it follows from Lemma \ref{yl1} that there exists a clopen neighborhood $V_{k+1}$ of $g_{k+1}$ such that $V_{k+1}\subset G\setminus (\bigcup_{i=0}^{k}V_{i})$ and $G=\langle G\setminus (\bigcup_{i=0}^{k+1}S_{i})\rangle$. Then (i)-(iii) are all satisfied.

Therefore, the sets $S_{n}$ and $V_{n}$ are defined for all $n$ with the required properties.
\end{proof}

\begin{corollary}\cite{CMRS1998}
Every countable Hausdorff topological group $G$ belongs to $\mathcal{S}_{cg}$.
\end{corollary}

\section{A strongly topological gyrogroup with a suitable set}
In this section, we mainly prove that every separable metrizable strongly topological gyrogroup has a suitable set. First, we need some lemmas.

\begin{lemma}\label{yl2}
Suppose that $(G,\tau ,\oplus)$ is a strongly topological gyrogroup with a symmetric neighborhood base $\mathscr U$ at $0$. Suppose further that $U,V,W$ are all open neighborhoods of $0$ such that $V\oplus V\subset W$, $W\oplus W\subset U$ and $V, W\in \mathscr U$. If a subset $A$ of $G$ is $U$-disjoint
(that is, if $b\not\in a\oplus U$, for any distinct $a, b\in  A$), then for each $x\in G$ the set
$x\oplus V$ intersects at most one of the element of the family $\{a\oplus V: a\in A\}$. In particular,
the family of open sets $\{a\oplus V:a\in A\}$ is discrete in $G$.
\end{lemma}

\begin{proof}
We need to show that, for every $x\in G$, the open neighborhood $x\oplus V$ of $x$ intersects at most one element of the family $\{a\oplus V: a\in A\}$. We assume the contrary that, for some $x\in G$, there exist distinct elements $a,b\in A$ such that $(x\oplus V)\cap (a\oplus V)\not=\emptyset$ and $(x\oplus V)\cap (b\oplus V)\not=\emptyset$. We show that $b\in a\oplus U$ as follows.

Since $(x\oplus V)\cap (a\oplus V)\not=\emptyset$, we have that there exist $v_{1},v_{2}\in V$ such that $x\oplus v_{1}=a\oplus v_{2}$. Then, $a=(a\oplus v_{2})\oplus \gyr[a,v_{2}](\ominus v_{2})=(x\oplus v_{1})\oplus \gyr[a,v_{2}](\ominus v_{2})$. Therefore,
\begin{eqnarray}
a&\in &(x\oplus v_{1})\oplus \gyr[a,v_{2}](V)\nonumber\\
&=&(x\oplus v_{1})\oplus V\nonumber\\
&=&x\oplus (v_{1}\oplus \gyr[v_{1},x](V))\nonumber\\
&=&x\oplus (v_{1}\oplus V)\nonumber\\
&\subset &x\oplus (V\oplus V)\nonumber\\
&\subset &x\oplus W.\nonumber
\end{eqnarray}

Thus, $a\in x\oplus W$. By the same method, we also have $b\in x\oplus W$.

Therefore, there exists $w_{1}\in W$ such that $a=x\oplus w_{1}$. Then, $$x=a\oplus \gyr[x,w_{1}](\ominus w_{1})\in a\oplus \gyr[x,w_{1}](W)=a\oplus W.$$ Hence,
\begin{eqnarray}
b&\in &(a\oplus W)\oplus W\nonumber\\
&=&a\oplus (W\oplus \gyr[W,a](W))\nonumber\\
&=&a\oplus (W\oplus W)\nonumber\\
&\subset &a\oplus U.\nonumber
\end{eqnarray}
\end{proof}

Let $G$ be a topological gyrogroup. For $\kappa$ an infinite cardinal,
the topological gyrogroup $G$ is said to be {\it left $\kappa$-totally bounded} if for every nonempty open subset
$U$ of $G$ there is $F\subset G$ such that $|F|<\kappa$ and $G=F\oplus U$. We denote $lb(G)$ by the least cardinal $\kappa\geq\omega$ such that $G$ is left $\kappa$-totally bounded. Each left $\omega$-totally bounded topological gyrogroup is also called {\it left precompact}.

\begin{lemma}\label{lll}
Let $G$ be a strongly topological gyrogroup with $lb(G)=\kappa$. If $\tau<\kappa$, then there exist an open neighborhood $V$ of $0$ and a subset $\{p_{\alpha}: \alpha<\tau\}$ such that for each $p\in G$ the set $p\oplus V$ intersects at most one of the elements of the family $\{p_{\alpha}\oplus V: \alpha<\tau\}$.
\end{lemma}

\begin{proof}
Since $lb(G)=\kappa$ and $\tau<\kappa$, it follows that there exists a nonempty open neighborhood $U$ of $0$ in $G$ such that no $F\subset G$ with $|F|\leq\tau$ satisfies $G=F\oplus U$. By induction, it is easy to find a set
$\{p_{\alpha}: \alpha<\tau\}$ such that each $p_{\alpha}$ satisfies $p_{\alpha}\not\in\bigcup_{\beta<\alpha}(p_{\beta}\oplus U)$. Then, from Lemma~\ref{yl2} we can find a nonempty open neighborhood $V$ of $0$ in $G$ such that for each $p\in G$ the set $p\oplus V$ intersects at most one of the elements of the family $\{p_{\alpha}\oplus V: \alpha<\tau\}$.
\end{proof}

The strongly topological gyrogroup $G$ in Example~\ref{lz1} is left precompact and non-pseudocompact. However, the following result shows that each pseudocompact strongly topological gyrogroup is left precompact.

\begin{theorem}\label{ttt}
Suppose that $(G, \tau, \oplus)$ is a strongly topological gyrogroup with a symmetric open neighborhood base $\mathscr U$ at $0$. If $G$ is pseudocompact, then it will be left precompact.
\end{theorem}

\begin{proof}
Let $U$ be an arbitrary symmetric open neighborhood of $0$ in $G$ and $V,W\in \mathscr U$ such that $V\oplus V\subset W$ and $W\oplus W\subset U$. Let $$\mathscr F=\{A\subset G:(b\oplus V)\cap (a\oplus V)=\emptyset,\mbox{ for any distinct}\ a, b\in A\}.$$ Define $\leq $ in $G$ such that $A_{1}\leq A_{2}$ if and only if $A_{1}\subset A_{2}$, for any $A_{1},A_{2}\in \mathscr F$. Then, $(\mathscr F,\leq)$ is a poset and the union of any chain of $V$-disjoint sets is again a $V$-disjoint set. Therefore, it follows from Zorn's Lemma that there exists a maximal element $A$ in $\mathscr F$ so that $\{a\oplus V:a\in A\}$ is a disjoint family of non-empty open sets in $G$. By Lemma \ref{yl2}, the family of open sets $\{a\oplus V:a\in A\}$ is discrete in $G$. What's more, $G$ is pseudocompact, we have that $A$ is finite. Finally, we show that $A\oplus U=G$ as follows.

Take an arbitrary $x\in G$. If $x\not\in A$, then it follows from the maximality of $A$ that there exists $a\in A$ such that $(x\oplus V)\cap (a\oplus V)\not =\emptyset$. Then, there exist $v_{1},v_{2}\in V$ such that $x\oplus v_{1}=a\oplus v_{2}$. By the right cancellation law, we have that
\begin{eqnarray}
x&=&(x\oplus v_{1})\oplus \gyr[x,v_{1}](\ominus v_{1})\nonumber\\
&=&(a\oplus v_{2})\oplus \gyr[x,v_{1}](\ominus v_{1})\nonumber\\
&\in &(a\oplus v_{2})\oplus \gyr[x,v_{1}](V)\nonumber\\
&=&(a\oplus v_{2})\oplus V\nonumber\\
&=&a\oplus (v_{2}\oplus \gyr[v_{2},a](V))\nonumber\\
&=&a\oplus (v_{2}\oplus V)\nonumber\\
&\subset &a\oplus (V\oplus V)\nonumber\\
&\subset &a\oplus U.\nonumber
\end{eqnarray}
Therefore, $A\oplus U=G$.
\end{proof}

\begin{theorem}\label{dl1}
Suppose that $(G,\tau ,\oplus)$ is a strongly topological gyrogroup with a symmetric open neighborhood base $\mathscr U$ at $0$, and that $H$ is an open $L$-subgyrogroup of $G$. If $H$ has a suitable set, then $G$ has a suitable set. If $H$ has a closed suitable set, then $G$ has a closed suitable set.
\end{theorem}

\begin{proof}
Let $S$ be a suitable set for $H$. Since $H$ is a $L$-subgyrogroup, two distinct cosets of $H$ are disjoint. Then let $A$ select one point from each coset of $H$ in $G$ such that $0\not\in A$ and $|A\cap (g\oplus H)|=1$ for each $x\in G$. We claim that $S\cup A$ is suitable for $G$.

Indeed, $S\cup\{0\}$ and $H$ are all closed in $G$, thus $S\cup A$ is discrete in $G\setminus\{0\}$, then there is at most an accumulation point $0$ since $S\cup A\cup\{0\}$ is closed in $G$. Now it suffices to prove that $\langle S\cup A\rangle$ is dense in $G$. Since $\langle S\rangle$ is dense in $H$, the subgyrogroup $\langle S\cup A\rangle$ is dense in $G$. If $S$ is closed in $H$, then $S\cup A$ is closed in $G$.
\end{proof}

However, the following question is open.

\begin{question}\label{q00}
Suppose that $(G,\tau ,\oplus)$ is a strongly topological gyrogroup, and that $H$ is an open subgyrogroup of $G$ with a suitable set. Does $H$ have a suitable set?
\end{question}

The following lemma gives a partial answer to Question~\ref{q00}.

\begin{lemma}\label{llll}
Suppose that $(G,\tau ,\oplus)$ is a separable strongly topological gyrogroup with a symmetric open neighborhood base $\mathscr U$ at $0$, $H$ is an open subgyrogroup of $G$. If $H$ has a suitable set, then $G$ has a suitable set. If $H$ has a closed suitable set, then $G$ has a closed suitable set.
\end{lemma}

\begin{proof}
Let $S$ be a suitable set for $H$. Since $G$ is separable, there exists a countable subset $A=\{g_{n}: n\in\omega\}$ of $G$ such that $g_{0}=0$ and $\overline{A}=G$, then $A\oplus H=G$ since $H$ is open in $G$. Then, by induction on $n$, we can choose a subset $B$ of $A$ satisfies the following conditions:

\smallskip
(i) $B$ is closed discrete;

\smallskip
(ii) $\overline{\langle B\cup S\rangle}=G$;

\smallskip
(iii) $B\cap (g\oplus H)=\{g\}$.

\smallskip
Indeed, take $g_{0}=\{0\}$. If $H=G$, then let $B=\{0\}$; otherwise, $G\setminus H\neq\emptyset$, since $G\setminus H$ is open, there exists a minimum $n_{1}\in\mathbb{N}$ such that $g_{n_{1}}\in(g_{n_{1}}\oplus H)\setminus H$ and $g_{i}\in H$ for any $i<n_{1}$. Assume have defined the points $g_{0}, g_{n_{1}}, \cdots, g_{n_{k}}$ such that $g_{n_{i}}\in (g_{n_{i}}\oplus H)\setminus\bigcup_{j<i}(g_{n_{j}}\oplus H)$ for each $i\leq k$ and $g_{j}\in\bigcup_{i=0}^{k-1}(g_{n_{i}}\oplus H)$ for any $n_{m}\leq j<n_{m+1}$ and $m\leq k-1$. If $\bigcup_{i=0}^{k}(g_{n_{i}}\oplus H)=G$, let $B=\{g_{n_{i}}: i\leq k\}$; otherwise, the set $G\setminus\bigcup_{i\leq k}(g_{n_{i}}\oplus H)$ is a nonempty open subset of $G$, then there exists a minimum $n_{k+1}\in\mathbb{N}$ such that $g_{n_{k+1}}\in (g_{n_{k+1}}\oplus H)\setminus\bigcup_{i\leq k}(g_{n_{i}}\oplus H)$ and $g_{j}\in\bigcup_{i\leq k}(g_{n_{i}}\oplus H)$ for each $j\leq n_{k+1}$. If there exists $N\in\mathbb{N}$ such that $\bigcup_{i\leq N}(g_{n_{i}}\oplus H)=G$, then $B=\{g_{n_{i}}: i\leq N\}$ is a finite set; otherwise, put $B=\{g_{n_{i}}: i\in\omega\}$. By our construction of $B$, it is easy to see that $B$ satisfies the conditions (i)-(iii).

By (ii), $\langle S\cup B\rangle$ is dense in $G$. Moreover, $S\cup\{0\}$ and $H$ are all closed in $G$, thus $(S\cup A)\setminus\{0\}$ is discrete in $G\setminus\{0\}$, then there is at most an accumulation point $0$ since $S\cup A\cup\{0\}$ is closed in $G$. If $S$ is closed in $H$, then $S\cup A$ is closed in $G$.
\end{proof}

\begin{lemma}\label{yl4}
Suppose that $(G,\tau ,\oplus)$ is a strongly topological gyrogroup with a symmetric open neighborhood base $\mathscr U$ at $0$, $B$ is a left precompact subset of $G$ and $S$ is dense in $B$. Then, for every neighborhood $U$ of $0$ in $G$, there is a finite set $K\subset S$ such that $B\subset K\oplus U$.
\end{lemma}

\begin{proof}
We assume that $U$ is an arbitrary neighborhood of $0$ in $G$ and $V\in \mathscr U$ such that $V\oplus V\subset U$. Since $B$ is left precompact in $G$, there exists a finite set $F$ in $G$ such that $B\subset F\oplus V$. Take an arbitrary $x\in F$ such that $B\cap (x\oplus V)\not =\emptyset$. Then $S\cap (x\oplus V)\not =\emptyset$ and we pick a point $y_{x}\in S\cap (x\oplus V)$. Then the finite set $$K_{1}=\{y_{x}:x\in F\mbox{ and } B\cap (x\oplus V)\not =\emptyset\}$$ is contained in $S$. We claim $B\subset K_{1}\oplus U$.

Indeed, if $b\in B$, then there exists $x\in F$ such that $b\in x\oplus V$, so $b\in B\cap (x\oplus V)\not =\emptyset$. Therefore, $y_{x}\in x\oplus V$. We can find $v_{1}\in V$ such that $y_{x}=x\oplus v_{1}$. Then
\begin{eqnarray}
x&=&(x\oplus v_{1})\oplus \gyr[x,v_{1}](\ominus v_{1})\nonumber\\
&=&y_{x}\oplus \gyr[x,v_{1}](\ominus v_{1})\nonumber\\
&\in &y_{x}\oplus \gyr[x,v_{1}](V)\nonumber\\
&=&y_{x}\oplus V.\nonumber
\end{eqnarray}
Thus,
\begin{eqnarray}
b&\in &x\oplus V\nonumber\\
&\subset &(y_{x}\oplus V)\oplus V\nonumber\\
&=&y_{x}\oplus (V\oplus \gyr[V,y_{x}](V))\nonumber\\
&=&y_{x}\oplus (V\oplus V)\nonumber\\
&\subset &y_{x}\oplus U\nonumber\\
&\subset &K_{1}\oplus U.\nonumber
\end{eqnarray}
\end{proof}

\begin{lemma}\label{yl5}
Every subgyrogroup $H$ of a left precompact strongly topological gyrogroup $G$ is left precompact.
\end{lemma}

\begin{proof}
Take an arbitrary open neighborhood $U$ of $0$ in $H$, then there is an open neighborhood $V$ of $0$ in $G$ such that $V\cap H=U$. Since $H$ is a left precompact subset of $G$, by Lemma \ref{yl4}, we can find a finite set $F\subset H$ such that $H\subset F\oplus V$. Therefore, for every $h\in H$, there exist $f\in F$ and $v\in V$ such that $h=f\oplus v$. Thus, $v=(\ominus f)\oplus h\in H\oplus H\subset H$. Then, $v\in V\cap H=U$. It follows that $H\subset F\oplus U$, that is, $H=F\oplus U$.
\end{proof}

By Theorem~\ref{ttt} and Lemma~\ref{yl5}, we have the following corollary.

\begin{corollary}
Every subgyrogroup $H$ of a pseudocompact strongly topological gyrogroup $G$ is left precompact.
\end{corollary}

\begin{lemma}\label{yl3}
Suppose that $(G,\tau ,\oplus)$ is a strongly topological gyrogroup with a symmetric open neighborhood base $\mathscr U$ at $0$, and suppose that $G$ is non-pseudocompact left precompact with a countable dense subgyrogroup $P$. Then there exists a subset $L\subset P$ such that $L$ is closed discrete in $G$ and $\langle L\rangle =P$. In particular, $L$ is suitable for $G$.
\end{lemma}

\begin{proof}
Since $G$ is not pseudocompact, we can fix a sequence $\{U_{n}:n\in \omega\}$ of non-empty open subsets of $G$ such that $U_{n}\in \mathscr U$, ~$\overline{U_{n+1}}\subset U_{n}$ for each $n\in \omega$ and $\bigcap \{U_{n}:n\in \omega\}=\emptyset$. Let $\{x_{n}:n\in \omega \}$ be an enumeration of elements of $P$.

We construct an increasing sequence $\{L_{k}: k\in \omega\}$ of finite subsets of $P$ by induction which satisfies the following conditions:

\smallskip
(1) $x_{k}\in \langle L_{k}\rangle$;

\smallskip
(2) $L_{k+1}\setminus L_{k}\subset U_{k}$;

\smallskip
(3) $G=\langle L_{k}\rangle \oplus U_{k}$.

Since the subgyrogroup $P$ is dense in $G$, it follows from \cite[Lemma 9]{AW} that $G=\overline{P} \subset P\oplus U_{0}$. So $G=P\oplus U_{0}$. Since $G$ is left precompact, it follows from Lemma \ref{yl4} that we can find a finite subset $K_{0}$ of $P$ such that $K_{0}\oplus U_{0}=G$. Therefore, for any $x_{0}\in G$, there exist $a_{0}\in K_{0},u_{0}\in U_{0}$ such that $x_{0}=a_{0}\oplus u_{0}$. Then $u_{0}=(\ominus a_{0})\oplus x_{0}\in P$ and Let $L_{0}=K_{0}\cup \{u_{0}\}$.

We assume that for some $n\in \omega$ we have defined an increasing sequence $L_{0},\ldots ,L_{n}$ of finite subsets of $P$ which satisfies (1) and (3) for each $k\leq n$ and (2) for every $k<n$. Since $P$ is dense in $G$, it is clear that $\langle U_{n}\cap P\rangle$ is dense in the gyrogroup $G_{n}=\langle U_{n}\rangle$. Thus, $\langle U_{n}\cap P\rangle \oplus U_{n+1}=G_{n}$. It follows from Lemma \ref{yl5} that $G_{n}$ is left precompact, hence we can find a finite subset $F_{n+1}$ of $\langle U_{n}\cap P\rangle$ such that $F_{n+1}\oplus U_{n+1}=G_{n}$. Clearly, we can find a finite subset $K_{n+1}$ of $U_{n}\cap P$ with $F_{n+1}\subset \langle K_{n+1}\rangle$, so $\langle K_{n+1}\rangle \oplus U_{n+1}=G_{n}$. Let $L_{n+1}^{'}=L_{n}\cup K_{n+1}$. By (3), we have
\begin{eqnarray}
G&=&\langle L_{n}\rangle \oplus U_{n}\nonumber\\
&\subset &\langle L^{'}_{n+1}\rangle \oplus G_{n}\nonumber\\
&=&\langle L^{'}_{n+1}\rangle \oplus (\langle K_{n+1}\rangle \oplus U_{n+1})\nonumber\\
&=&(\langle L^{'}_{n+1}\rangle \oplus \langle K_{n+1}\rangle )\oplus \gyr[\langle L^{'}_{n+1}\rangle,\langle K_{n+1}\rangle](U_{n+1})\nonumber\\
&=&(\langle L^{'}_{n+1}\rangle \oplus \langle K_{n+1}\rangle)\oplus U_{n+1}\nonumber\\
&=&\langle L^{'}_{n+1}\rangle \oplus U_{n+1}.\nonumber
\end{eqnarray}
Therefore, there exist $a_{n+1}\in \langle L^{'}_{n+1}\rangle ,u_{n+1}\in U_{n+1}$ such that $x_{n+1}=a_{n+1}\oplus u_{n+1}$. Since $a_{n+1}\in \langle L^{'}_{n+1}\rangle \subset P$ and $x_{n+1}\in P$, it follows that $u_{n+1}=(\ominus a_{n+1})\oplus x_{n+1}\in P$. Then let $L_{n+1}=L_{n+1}^{'}\cup \{u_{n+1}\}$. It is clear that $L_{n+1}$ is a finite subset of $P$ and $L_{n}\subset L_{n+1}$. At the same time, $\langle L_{n+1}\rangle \oplus U_{n+1}=G$. Moreover, $L_{n+1}\setminus L_{n}\subset K_{n+1}\cup \{u_{n+1}\}\subset U_{n}$. Therefore, we complete the construction.

Finally, set $L=\bigcup \{L_{n}:n\in \omega\}$. It follows from (2) that $L\setminus \overline{U_{k}}\subset L_{k}$ is a finite set for each $k\in \omega$. Then, $\bigcap \{\overline{U_{n}}:n\in \omega\}$ implies that $L$ is a closed discrete subset of $G$. Moreover, (1) guarantees that $\langle L\rangle =P$.
\end{proof}

Now we can prove one of main results in this section.

\begin{theorem}\label{t11}
Suppose that $(G,\tau ,\oplus)$ is a non-pseudocompact strongly topological gyrogroup with a symmetric open neighborhood base $\mathscr U$ at $0$. If $G$ is separable, then $G\in\mathcal{S}_{c}$.
\end{theorem}

\begin{proof}
It follows from \cite{BL1} that $G$ is Tychonoff. We divide the proof into two cases:

\smallskip
{\bf Case 1:} $G$ is not left precompact.

\smallskip
Then $G$ has a neighborhood $U$ of $0$ such that $F\oplus U\not =G$ for any finite $F\subset G$. By Lemma \ref{yl2}, we need to take $V, W\in \mathscr U$ such that $V\oplus V\subset W$ and $W\oplus W\subset U$, then there exists a subset $A=\{a_{n}:n\in \omega\}\subset G$ with $a_{i}\not =a_{j}$ if $i\not =j$ and the family $\gamma =\{a_{n}\oplus V:n\in \omega\}$ is discrete in $G$. Let $B$ a countable dense subset of $G$, and set $B_{V}=B\cap V=\{b_{n}: n\in\omega\}$. We prove that $S=A\cup (\bigcup _{n\in\mathbb{N}}a_{n}\oplus b_{n})$ is a suitable set of open subgyrogroup $G_{1}=\langle V\cup A\rangle$.

Since $\langle B_{V}\rangle$ is dense in $\langle V\rangle$ and $B_{V}\subset \langle S\rangle$, we have that $\langle S\rangle$ is dense in $G_{1}$. For every $g\in G$, there exists a neighborhood $O$ of $0$ in $G$ such that $(g\oplus O)\cap S\subset\{a_{n}, a_{n}\oplus b_{n}\}$. Therefore, $S$ is closed and discrete and hence it is a suitable set of $G_{1}$. Then since $G_{1}$ is an open subgyrogroup of $G$, it follows from Lemma~\ref{llll} that $G$ has a closed suitable set.

\smallskip
{\bf Case 2:} $G$ is left precompact.

\smallskip
Since $G$ is non-pseudocompact, we can choose a discrete family $\gamma =\{U_{n}:n\in \omega \}$ of non-empty open subsets of $G$. Let $B=\{d_{n}: n\in\mathbb{N}\}$ be a countable dense subset of $G$.

Since $G$ is precompact, for every $n\in \omega$, there exists a finite subset $A_{n}=\{a(n, i):1\leq i\leq m_{n}\}$ of $G$ such that $A_{n}\oplus U_{n}=G$. Fix an $n\in \omega$ and define $H_{n}^{i}=\{d_{n}\}\cap (a(n, i)\oplus U_{n})$ for each $i\leq m_{n}$. Then $H_{n}=\bigcup \{H_{n}^{i}:1\leq i\leq m_{n}\}$.

The set $T_{n}=\bigcup \{(\ominus a(n, i))\oplus H_{n}^{i}:1\leq i\leq m_{n}\}$ is closed and discrete in $G$ and lies in $U_{n}$. Since the family $\gamma$ is discrete, the set $T=\bigcup \{T_{n}: n\in \omega\}$ is closed and discrete in $G$. Let $A=\bigcup \{A_{n}: n\in \omega\}$. For every $n\in \omega$, choose a point $y_{n}\in U_{n}$ such that $d_{n}\in A_{n}\oplus y_{n}$ and denote by $G_{2}$ the closure of $P=\langle A\cup \{y_{n}: n\in \omega\}\rangle$ in $G$. The gyrogroup $G_{2}$ is closed and left precompact by Lemma~\ref{yl5}. Moreover, for each $n\in \omega$, $G_{2} \cap U_{n}\neq\emptyset$, so $G_{2}$ is not pseudocompact.

It follows from Lemma \ref{yl3} that there is a closed discrete subset $L$ of $G_{2}$ such that $\langle L\rangle =P$. We find that $T\cup L$ is closed and discrete in $G$ and $P\subset \langle T\cup L\rangle \supset \langle T\cup A \rangle\supset B$. Hence $\langle T\cup L\rangle$ is dense in $G$. Therefore, $T\cup L$ is a closed suitable set for $G$.
\end{proof}

\begin{lemma}\label{tt17}
Let $G$ be a compact metrizable strongly topological gyrogroup. Then $G$ has a closed suitable set.
\end{lemma}

\begin{proof}
Since $G$ is compact metrizable, it is separable, hence there exists a countable dense subgroup $P$. Let $P=\{x_{n}: n\in\omega\}$ be a enumeration of $P$. Moreover, we can choose a decreasing sequence $\{U_{n}: n\in\omega\}$ of open neighborhoods of the identity $0$ in $G$
satisfying the following conditions:

\smallskip
(1) $U_{n+1}\oplus U_{n+1}\subset U_{n}$ for each $n\in\omega$;

\smallskip
(2) $\bigcap_{n\in\omega}U_{n}=\{0\}$.

\smallskip
By the same construction of Lemma~\ref{yl3}, we can find an increasing sequence $\{L_{k}:k\in \omega\}$ of finite subsets of $P$ by induction which satisfies the following conditions:

\smallskip
(1) $x_{k}\in \langle L_{k}\rangle$;

\smallskip
(2) $L_{k+1}\setminus L_{k}\subset U_{k}$;

\smallskip
(3) $G=\langle L_{k}\rangle \oplus U_{k}$.

\smallskip
By a similar proof of Lemma~\ref{yl3}, we can find a closed discrete subset $L$ for $P$. Then $L$ is a closed suitable set for $G$ since $P$ is dense in $G$.
\end{proof}

A space $X$ is {\it paracompact}, if each its open cover has a locally finite open refinement.
A space $X$ is {\it submetrizable}, if there exists a continuous injective map of $X$ to
a metrizable space.

\begin{corollary}\label{y16}
 Suppose that $(G,\tau ,\oplus)$ is a separable left precompact Hausdorff strongly topological gyrogroup of countable pseudocharacter with a symmetric open neighborhood base $\mathscr U$ at $0$. If $P$ is a countable dense subgyrogroup of $G$, then there exists a discrete subset $L$ of $P$ such that $L$ is closed in $G\setminus \{0\}$ and $P=\langle L\rangle$. So $L$ is a suitable set for both $P$ and $G$.
\end{corollary}
\begin{proof}
If $G$ is not pseudocompact, then it follows from Lemma~\ref{yl3} that the conclusion holds. Assume that $G$ is pseudocompact, then from \cite{BL, BL1} that each strongly topological gyrogroup of countable pseudocharacter is paracompact
and submetrizable, hence it is compact and metrizable, thus $G$ has a closed suitable set by Lemma~\ref{tt17}.
\end{proof}

A space $X$ is a {\it $\sigma$-space} if it has a $\sigma$-locally finite network.

\begin{corollary}\label{c0}
 Suppose that $(G,\tau ,\oplus)$ is a strongly topological gyrogroup. If $G$ is a separable $\sigma$-space then $G$ has a suitable set.
\end{corollary}

\begin{proof}
Since each $\sigma$-space has a countable pseudocharacter, $G$ has countable pseudocharacter. If $G$ is not
pseudocompact, the conclusion holds from Theorem~\ref{t11}. From \cite{BL, BL1}, each strongly topological gyrogroup of countable pseudocharacter is paracompact and submetrizable, hence it is compact and metrizable, thus separable precompact,  so we can apply Lemma~\ref{y16} to conclude that $G$ has a suitable set.
\end{proof}

\begin{corollary}\label{c1}
 Suppose that $(G,\tau ,\oplus)$ is a strongly topological gyrogroup. If $G$ is a separable metrizable space, then $G$ has a suitable set.
\end{corollary}

We now close our paper with the following three questions.

\begin{question}
 Suppose that $(G,\tau ,\oplus)$ is a metrizable strongly topological gyrogroup, does $G$ have a suitable set?
\end{question}

\begin{question}
Does each locally compact (strongly) topological gyrogroup have a suitable set? What if the space is compact?
\end{question}

\begin{question}\label{dl5.7}
Suppose that $(G,\tau ,\oplus)$ is a Hausdorff strongly topological gyrogroup with a symmetric open neighborhood base $\mathscr U$ at $0$ which satisfies $d(G)< b(G)$, does $G$ have a closed suitable set?
\end{question}

{\bf Acknowledgments.} The authors thank to Alex Ravsky for valuable remarks and corrections and all other sort of help related to the content of this paper.


\begin{thebibliography}{33}

\bibitem{AA} A.V. Arhangel'ski\v\i, M. Tkachenko, {\it Topological Groups and Related Structures}, Atlantis Press and World Sci., 2008.

\bibitem{AW} W. Atiponrat, {\it Topological gyrogroups: generalization of topological groups}, Topol. Appl., {\bf 224} (2017) 73--82.

\bibitem{BP2003} T. Banakh, I. Protasov,
{\it Ball Structures and Colorings of Graphs and Groups},
VNTL Publ., Lviv, 2003.

\bibitem{BL} M. Bao, F. Lin, {\it Feathered gyrogroups and gyrogroups with countable pseudocharacter}, Filomat, {\bf 33}(16) (2019) 5113--5124.

\bibitem{BL1} M. Bao, F. Lin, {\it Submetrizability of strongly topological gyrogroups}, arxiv.org/abs/2003.06132.

\bibitem{BL2} M. Bao, F. Lin, {\it Quotient with respect to admissible L-subgyrogroups}, arxiv.org/abs/2003.08843v2.

\bibitem{CZ} Z. Cai, S. Lin, W. He, {\it A note on Paratopological Loops}, Bulletin of the Malaysian Math. Sci. Soc., {\bf 42(5)} (2019) 2535--2547.

\bibitem{CMRS1998} W.W. Comfort, S.A. Morris, D. Robbie, S. Svetlichny, M. Tkachenko, {\it Suitable sets for topological groups}, Topol. Appl., {\bf 86} (1998) 25--46.

\bibitem{DTT1999} D. Dikranjan, M. Tkachenko, V. Tkachuk,
{\it Some topological groups with and some without suitable sets}, Topol. Appl., {\bf 98} (1999) 131--148.

\bibitem{DTT2000} D. Dikranjan, M. Tka\v{c}enko, V. Tkachuk,
{\it Topological groups with thin generating sets},
J. Pure Appl. Algebra, {\bf 145}:2 (2000), 123--148.

\bibitem{E} R. Engelking, General Topology(revised and completed edition), Heldermann Verlag,
Berlin, 1989.

\bibitem{FM} M. Ferreira, {\it Factorizations of M\"{o}bius gyrogroups}, Adv. Appl. Clifford Algebras, {\bf 19} (2009) 303--323.

\bibitem{FM1} M. Ferreira, G. Ren, {\it M\"{o}bius gyrogroups: A Clifford algebra approach}, J. Algebra, {\bf 328} (2011) 230--253.

\bibitem{G2003}
I. Guran,
{\it Suitable sets for paratopological groups},
Abstracts of 4-th International Algebraic Conference in Ukraine (Lviv, 2003), 87--88.


\bibitem{HM} K.H. Hoffmann, S.A. Morris, {\it Weight and $\mathfrak{c}$}, J. Pure Appl. Algebra, {\bf 68} (1990) 181--194.

\bibitem{LF} F. Lin. R. Shen, {\it On rectifiable spaces and paratopological groups}, Topol. Appl., {\bf 158} (2011) 597--610.

\bibitem{LF1} F. Lin. C. Liu, S. Lin, {\it A note on rectifiable spaces}, Topol. Appl., {\bf 159} (2012) 2090--2101.

\bibitem{LR} F. Lin, A. Ravsky, {\it Suitable sets for paratopological groups},
\url{arXiv:2005.08233v1}

\bibitem{LF2} F. Lin, {\it Compactly generated rectifiable spaces or paratopological groups}, Math. Commun., {\bf 18} (2013) 417--427.

\bibitem{LF3} F. Lin, J. Zhang, K. Zhang, {\it Locally $\sigma$-compact rectifiable spaces}, Topol. Appl., {\bf 193} (2015) 182--191.

\bibitem{SL} L.V. Sabinin, L.L. Sabinin, L.V. Sbitneva, {\it On the notion of gyrogroup}, Aequ. Math., {\bf 56} (1998) 11--17.

\bibitem{RS} R. Shen, {\it The first-countability of paratopological left-loops}, Topol. Appl., Article 107190.

\bibitem{ST} T. Suksumran, K. Wiboonton, {\it Isomorphism theorems for gyrogroups and $L$-subgyrogroups}, J. Geom. Symmetry Phys., {\bf 37} (2015) 67--83.

\bibitem{ST1} T. Suksumran, {\it Essays in mathematics and its applications: in honor of Vladimir Arnold, in: P.M. Pardalos, T.M. Rassias (Eds.), The Algebra of Gyrogroups: Cayley's Theorem, Lagrange's Theorem, and Isomorphism Theorems}, Springer, 2016, 369--437.

\bibitem{ST2} T. Suksumran, {\it Special subgroups of gyrogroups: commutators, nuclei and radical}, Math. Interdiscip. Res, {\bf 1} (2016) 53--68.

\bibitem{UA} A.A. Ungar, {\it Analytic Hyperbolic Geometry and Albert Einstein's Special Theory of Relativity}, World Scientific, Hackensack, New Jersey, 2008.

\bibitem{UA2002} A.A. Ungar,{\it  Beyond the Einstein addition law and its gyroscopic Thomas precession: The theory
of gyrogroups and gyrovector spaces}, Fundamental Theories of Physics, vol. 117, Springer, Netherlands, 2002.
\end{thebibliography}
\end{document}